\documentclass[11pt]{article}

\usepackage[margin=1in]{geometry}
\usepackage{amsthm}
\usepackage{amssymb}
\usepackage{amsmath}
\usepackage{graphicx}
\usepackage{url}
\usepackage{color}

\newcommand\blfootnote[1]{%
  \begingroup
  \renewcommand\thefootnote{}\footnote{#1}%
  \addtocounter{footnote}{-1}%
  \endgroup
}

\newtheorem{theorem}{Theorem}

\theoremstyle{definition}

\newtheorem{definition}[theorem]{Definition}

\title{An impossibility theorem for gerrymandering}

\author{Boris Alexeev\qquad Dustin~G.~Mixon\footnote{Department of Mathematics, The Ohio State University, Columbus, OH}}
\date{}

\begin{document}
\maketitle

\blfootnote{Send correspondence to \texttt{mixon.23@osu.edu}}

\begin{abstract}
The U.S.\ Supreme Court is currently deliberating over whether a proposed mathematical formula should be used to detect unconstitutional partisan gerrymandering.
We show that in some cases, this formula will only flag bizarrely shaped districts as potentially constitutional.
\end{abstract}


In 1812, the Boston Gazette published a political cartoon that likened the contorted shape of a Massachusetts state senate election district to the profile of a salamander~\cite{Griffith:07}.
The cartoon insinuated that Governor Elbridge Gerry approved this district's shape for his party's benefit, thereby coining the portmanteau ``gerrymander.''
Ever since, it has been common practice to use geometry as a signal for gerrymandering, with the most egregious districts exhibiting bizarre shapes; see \cite{Ingraham:14,Pisanty:online} for example.
To help bring this geometric signal to gerrymandering court cases, a team of Boston-based mathematicians known as the Metric Geometry and Gerrymandering Group is currently offering expert witness training in a sequence of Geometry of Redistricting workshops across the country~\cite{Metric:online}.

Partisan gerrymandering is the subject of a U.S.\ Supreme Court decision expected next year~\cite{CohnB:17}.
In this case, the justices will evaluate a completely different approach to detect gerrymandering.
Instead of flagging districts with irregular shapes, the proposed method attempts to detect the intended consequence of partisan gerrymandering:\ one party wasting substantially more votes than the other party.
This method summarizes the disproportion of wasted votes in a tidy statistic known as efficiency gap.

Recently, Bernstein and Duchin~\cite{BernsteinD:17} provided a helpful discussion of efficiency gap, in which they mention that it sometimes incentivizes bizarrely shaped districts.
This is perhaps counterintuitive considering the geometry-infused history of gerrymandering.
In this note, we demonstrate an extreme version of this observation:
\begin{center}
\textit{Sometimes, a small efficiency gap is only possible with bizarrely shaped districts.}
\end{center}
Specifically, we show that every districting system must violate one of three well-established desiderata that we make explicit later: one person, one vote; Polsby--Popper compactness; and partisan efficiency.
As such, our result is reminiscent of Arrow's impossibility theorem concerning ranked voting electoral systems~\cite{Arrow:50}.

\begin{definition}
A \textbf{districting system} is a function that receives disjoint finite sets $A,B\subseteq [0,1]^2$ and a positive integer $k$, and then outputs a partition $D_1\sqcup\cdots\sqcup D_k=[0,1]^2$.
\end{definition}

Here, $A$ and $B$ correspond to voter locations from two major parties, respectively; we do not consider third-party voters.
In practice, districts are drawn given the locations of the entire population from census data, but $A$ and $B$ can be estimated using past election data; in particular, these estimates enable partisan gerrymandering.
We focus on districting systems for the unit square largely for convenience, and without loss of generality.
Indeed, one may partition any state with such a districting system by first inscribing the state in a square, and conversely, a districting system on any state determines a system for the square by inscribing a square in that state.

When evaluating a given districting system $f\colon(A,B,k)\mapsto\{D_i\}_{i=1}^k$, one may test for any number of desirable characteristics.
What follows is a list of such characteristics.
Here, when we say ``the districts always satisfy'' a given condition, we mean that all possible choices of $(A,B,k)$ simultaneously allow for $\{D_i\}_{i=1}^k=f(A,B,k)$ to satisfy that condition.
\begin{itemize}
\item[(i)]
\textbf{One person, one vote.}
There exists $\delta\in[0,1)$ such that the districts always satisfy
\begin{equation}
\label{eq.1}
(1-\delta)\bigg\lfloor\frac{|A\cup B|}{k}\bigg\rfloor
\leq \big|(A\cup B)\cap D_i\big|
\leq (1+\delta)\bigg\lceil\frac{|A\cup B|}{k}\bigg\rceil
\qquad
\forall i\in\{1,\ldots,k\}.
\end{equation}
In words, the districts are drawn to contain roughly equal numbers of voters.
Assuming equal voter turnout, this is equivalent to the districts containing roughly equal represented populations.
The latter has been a guiding principle for all levels of redistricting in the United States following a series of U.S.\ Supreme Court decisions in the 1960s, namely \textit{Gray v.\ Sanders}, \textit{Reynolds v.\ Sims}, \textit{Wesberry v.\ Sanders}, and \textit{Avery v.\ Midland County}~\cite{Smith:14}.
\item[(ii)]
\textbf{Polsby--Popper compactness.}
There exists $C>0$ such that the districts always satisfy
\begin{equation}
\label{eq.2}
|\partial D_i|^2
\leq C|D_i|
\qquad
\forall i\in\{1,\ldots,k\}.
\end{equation}
Here, $|\partial D_i|$ denotes the perimeter of $D_i$, whereas $|D_i|$ denotes its area.
In 1991, Polsby and Popper~\cite{PolsbyP:91} introduced their so-called \textbf{Polsby--Popper score}, defined by $4\pi|D_i|/|\partial D_i|^2$, as a measure of geographic compactness.
Their intent was to allow for an enforceable standard (e.g., no district shall score below $0.2$ without additional scrutiny) that would ``make the gerrymanderer's life a living hell.''
In this spirit, Arizona's redistricting commission in 2000 used the Polsby--Popper score to ensure geographic compactness amongst their voting districts~\cite{Moncrief:11}.
The exceedingly long perimeters of the 1st and 12th congressional districts of North Carolina were cited in the recent U.S.\ Supreme Court case \textit{Cooper v.\ Harris}, in which the Court ruled that both districts were the result of unconstitutional racial gerrymandering~\cite{Blythe:17}.
At the time, these were two of the three congressional districts across the country with the smallest Polsby--Popper scores~\cite{Pisanty:online}.
\item[(iii)]
\textbf{Partisan efficiency.}
There exist $\alpha,\beta>0$ such that the districts always satisfy
\begin{equation}
\label{eq.3}
\big|\operatorname{EG}(D_1,\ldots,D_k;A,B)\big|<\frac{1}{2}-\alpha
\qquad
\text{whenever}
\qquad
\big||A|-|B|\big|<\beta|A\cup B|.
\end{equation}
Here, the so-called \textbf{efficiency gap} $\operatorname{EG}(\cdot;A,B)$ quantifies the extent to which votes are disproportionately ``wasted'' by the districting.
Suppose $|A\cap D_i|>|B\cap D_i|$.
Then the number of wasted votes in $A\cap D_i$ is the excess
\[
|A\cap D_i|-\bigg\lceil \frac{1}{2}\big|(A\cup B)\cap D_i\big|\bigg\rceil,
\]
considering $A$ did not need these votes to carry the district $D_i$.
Meanwhile, all of the votes in $B\cap D_i$ were wasted, since they did not contribute to winning any district.
Letting $w_{A,i}$ and $w_{B,i}$ denote the wasted votes in $A\cap D_i$ and $B\cap D_i$, respectively, then the efficiency gap is defined by
\[
\operatorname{EG}(D_1,\ldots,D_k;A,B)
:=\frac{1}{|A\cup B|}\sum_{i=1}^k\Big(w_{A,i}-w_{B,i}\Big).
\]
Stephanopoulos and McGhee introduced the efficiency gap in~\cite{StephM:15}, and it plays a key role in the U.S.\ Supreme Court case \textit{Gill v.\ Whitford}; the Court is currently deliberating over whether this gap should be used to signal unconstitutional partisan gerrymandering~\cite{CohnB:17}.
We note that the efficiency gap can range anywhere from $0$ to $1/2$, and Stephanopoulos and McGhee suggest that a gap of 8\% or more is sufficient to flag potential partisan gerrymandering.
Bernstein and Duchin~\cite{BernsteinD:17} observe that such a small efficiency gap is only possible when neither $A$ nor $B$ make up more than 79\% of the vote.
In particular, there already exist $A$ and $B$ for which no choice of districts satisfies \eqref{eq.3} with $\alpha=0.50-0.08=0.42$ and $\beta>0.79-(1-0.79)=0.58$.
By comparison, our notion of partisan efficiency is particularly weak since we allow $\alpha$ and $\beta$ to be arbitrarily small.
\end{itemize}

Observe that (i) and (ii) are agnostic to the voting preferences of $A$ and $B$; in particular, (i) only depends on $A\cup B$, whereas (ii) only sees the resulting districts.
Meanwhile, (iii) explicitly distinguishes between $A$ and $B$.
In practice, (iii) would be evaluated after election day, though it could be predicted with the help of past election data.
While (i), (ii) and (iii) have the form ``there exist constant(s) such that the districts always satisfy some condition,'' the court may assign values for $\delta$, $C$, $\alpha$ and $\beta$, explicitly requiring the districts to satisfy \eqref{eq.1}, \eqref{eq.2} and \eqref{eq.3} with these parameters.
The following theorem establishes that such a requirement would not always be feasible:

\begin{theorem}
\label{thm.main result}
There is no districting system that simultaneously satisfies all three desiderata above.
In particular, for every $\delta$, $C$, $\alpha$, $\beta$ and $k$, there exist $A$ and $B$ such that every choice of districts $\{D_i\}_{i=1}^k$ violates one of \eqref{eq.1}, \eqref{eq.2} and \eqref{eq.3}.
\end{theorem}

Notably, our result does not require the districts to be connected.
The idea behind the proof is straightforward:
We consider homogeneous mixtures of voters in the unit square, where just over half of the voters belong to $A$, and just under half belong to $B$.
In this extreme case, one would need to surgically design a district in order for $B$ to be the majority while simultaneously being large enough to satisfy (i).
This surgery would in turn force the district to exhibit bizarre shape, i.e., its Polsby--Popper score would be quite small (see Figure~\ref{fig.example}).
As such, districts satisfying both (i) and (ii) are necessarily majority-$A$.
All told, $A$ wastes a tiny portion of its votes, whereas $B$ wastes all of its votes, thereby violating (iii).

Before presenting the formal proof, we take a moment to discuss the result.
First, one could argue that our result does not necessarily preclude the real-world utility of (i), (ii) and (iii), as the $A$ and $B$ we construct are perhaps not likely arise in practice.
Indeed, our result does not suggest that impossibility arrives from every possible distribution of voters.
For example, if the western half of the state were all blue and the eastern half were all red, then it would be straightforward to draw nice-looking districts with small efficiency gap.
So now we have a spectrum of distribution possibilities (from purely homogeneous to completely separated), and the real-world partisan distributions reside somewhere along this spectrum.
How does impossibility depend on this spectrum, and where does reality reside?
These are both interesting questions that warrant further investigation. 

Along these lines, we point out that perhaps surprisingly, the issue is not necessarily resolved by the presence of partisan clusters.
For instance, you can modify the example in Figure~\ref{fig.example} by selecting a contiguous portion of the map and changing all the red votes in that region to blue.
While this would result in a blue partisan cluster, the modification would not change the fact that all districts in the middle panel are majority blue, and so the efficiency gap would still be quite large.
One could argue that the congressional districts in Massachusetts present a real-world example of this phenomenon, as all of the seats in this state are occupied by Democrats even though the districts do not exhibit bizarre shape.

While our impossibility result identifies a tension between Polsby--Popper compactness and partisan efficiency, we suspect there is a more general meta-theorem dictating a fundamental tradeoff between geographic compactness and simple quantifications of partisan gerrymandering.
For example, the $A$ and $B$ we construct demonstrates impossibility with any alternative to efficiency gap that would disallow a slight majority winning every district.
One such alternative is proportionality, which requires the number of seats won across the state to be roughly proportional to the number of votes cast; impossibility here is perhaps a moot point since the U.S.\ Supreme Court already established in \textit{Davis
v.\ Bandemer} that proportionality is not a valid constitutional standard.

Many states have laws requiring voting districts to exhibit geographic compactness~\cite{Levitt:online}, but there is no standard approach to measure compactness.
For example, as an alternative to Polsby--Popper, one could ask that a district's area be sufficiently large compared to either the smallest circle containing the district or the district's convex hull~\cite{Redrawing:online}.
It appears that the techniques in our proof do not easily transfer to these alternatives.
In particular, forcing the districts to be convex already presents what appears to be an interesting problem in additive combinatorics, which we leave for future work.

\begin{proof}[Proof of Theorem~\ref{thm.main result}]
Fix $k$.
For $n$ large, partition $[0,1]^2$ into squares of edge length $\epsilon=1/n$.
Take $a$, $b$ and $\ell$ such that $a+b=\ell^2$, and let $L$ denote the lattice $(\frac{1}{n\ell}(\mathbb{Z}+\frac{1}{2}))^2$.
Define $A$ to have $a$ voters in each $\epsilon$-square intersect $L$, and define $B$ to have $b$ voters in each $\epsilon$-square intersect $L$.
See Figure~\ref{fig.example} for an illustration.

Now take a partition into districts $D_1\sqcup\cdots\sqcup D_k=[0,1]^2$ that satisfies (i) and (ii).
Pick $i\in\{1,\ldots,k\}$.
The $\epsilon$-squares that contain $\partial D_i$ are in turn contained in an $\epsilon\sqrt{2}$-thickened version of $\partial D_i$, which has area at most $\sqrt{2}|\partial D_i|\epsilon+2\pi\epsilon^2$.
As such, $\partial D_i$ is contained in at most $E:=\sqrt{2}|\partial D_i|/\epsilon+2\pi$ different $\epsilon$-squares.
Meanwhile, $D_i$ contains at least $|D_i|/\epsilon^2-E$ and at most $|D_i|/\epsilon^2$ different $\epsilon$-squares.
Overall, we may conclude that $A$ wins the district $D_i$ if the second inequality below holds:
\[
|A\cap D_i|
\geq a\bigg(\frac{|D_i|}{\epsilon^2}-E\bigg)
\geq b\bigg(\frac{|D_i|}{\epsilon^2}+E\bigg)
\geq |B\cap D_i|.
\]
Specifically, it suffices to have
\[
\frac{b}{a}
\leq\frac{|D_i|-\epsilon^2 E}{|D_i|+\epsilon^2 E},
\]
which by (ii) is implied by
\begin{equation}
\label{eq.to get}
\frac{b}{a}
\leq\frac{|\partial D_i|^2-C\epsilon^2 E}{|\partial D_i|^2+C\epsilon^2 E}.
\end{equation}
Next, (i) gives that
\[
\big|(A\cup B)\cap D_i\big|
\geq (1-\delta)\bigg\lfloor\frac{|A\cup B|}{k}\bigg\rfloor
\geq (1-\delta)\frac{|A\cup B|}{2k}
=(1-\delta)\frac{n^2\ell^2}{2k},
\]
and since these points lie in $L$, two of them must be of distance at least $\sqrt{(1-\delta)/(2k)}$ from each other.
As such, $|\partial D_i|\geq\sqrt{(1-\delta)/(2k)}=:F$, and so \eqref{eq.to get} is implied by
\begin{equation}
\label{eq.bound}
\frac{b}{a}
\leq\frac{F^2-CF\sqrt{2}\epsilon-2C\pi\epsilon^2}{F^2+CF\sqrt{2}\epsilon+2C\pi\epsilon^2}.
\end{equation}
Observe that \eqref{eq.bound} is independent of our choice of $i$, and so \eqref{eq.bound} implies that $A$ wins every district $D_i$.

At this point, since $n$ was chosen to be arbitrarily large, $\epsilon$ is arbitrarily small, and so we may pick $a$ and $b$ so that $\gamma=1-b/a>0$ is arbitrarily small while still satisfying \eqref{eq.bound}.
As such, $B$ loses every district, thereby wasting all $bn^2=(1-\gamma)an^2$ of its votes, whereas $A$ narrowly wins every district, thereby wasting at most $an^2-bn^2=\gamma an^2$ of its votes.
Overall, the efficiency gap is
\[
\operatorname{EG}(D_1,\ldots,D_k;A,B)
\leq\frac{\gamma an^2-bn^2}{an^2+bn^2}
=\frac{2\gamma-1}{2-\gamma},
\]
which is arbitrarily close to $-1/2$ despite
\[
\frac{\big||A|-|B|\big|}{|A\cup B|}
=\frac{an^2-bn^2}{an^2+bn^2}
=\frac{\gamma}{2-\gamma}
\]
being arbitrarily small. 
Therefore, every districting system that satisfies both (i) and (ii) necessarily violates (iii).
\end{proof}

\begin{figure}
\centering
\includegraphics[width=\textwidth]{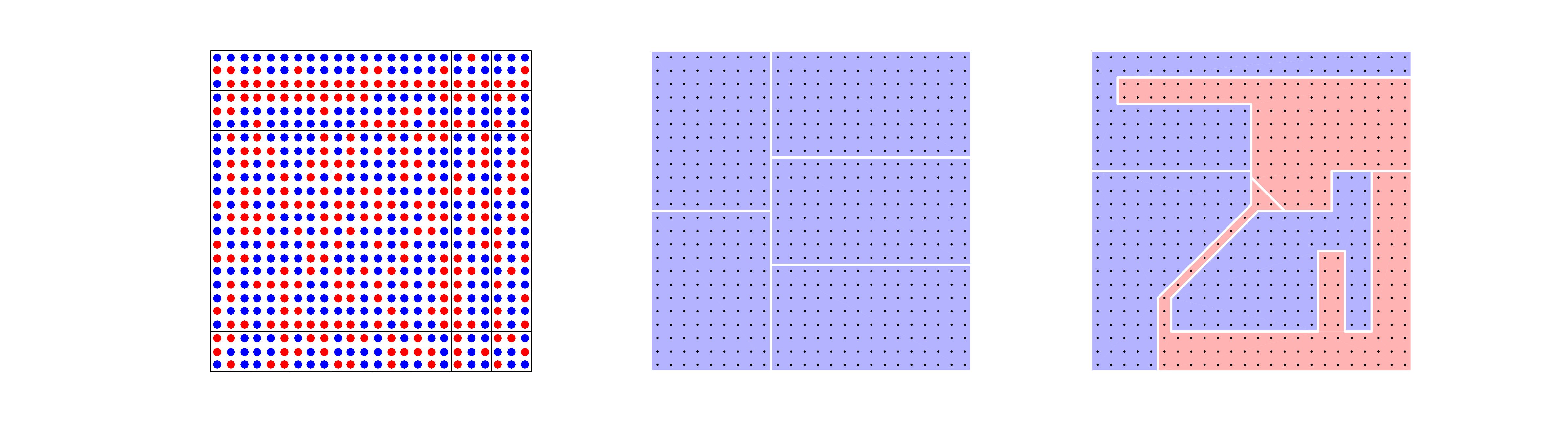}
\caption{\label{fig.example}
\textbf{(left)}
Voter locations in $[0,1]^2$.
The blue and red dots correspond to $A$ and $B$, respectively.
In this example, the unit square is partitioned into smaller squares of side length $\epsilon=1/8$.
Each $\epsilon$-square contains $9$ voters in an $\ell\times \ell$ subset of a lattice with $\ell=3$.
Of these voters, $a=5$ belong to $A$ and $b=4$ belong to $B$.
\textbf{(middle)}
Five districts drawn according to the shortest splitline algorithm proposed by the Center for Range Voting in~\cite{Gerry:online}.
This algorithm is specifically designed to produce districts that satisfy desiderata (i) and (ii).
In particular, this algorithm ignores voter preferences.
In this case, the number of voters in each district is within $\delta=0.07$ of the average and the smallest Polsby--Popper score $4\pi|D_i|/|\partial D_i|^2$ is over $0.70$.
(According to the isoperimetric inequality~\cite{Osserman:78}, the largest score possible is 1, which is achieved uniquely by the circle.)
However, despite $B$ making up 44\% of the vote, $A$ won every district.
This is reflected in the efficiency gap being over 38\% in favor of $A$.
\textbf{(right)}
One may attempt to decrease the efficiency gap by exploiting clusters in $B$.
In this spirit, we hand-drew districts of similar size.
The result is five districts within $\delta=0.04$ of the average and an efficiency gap of about 2\% in favor of $A$.
In exchange for this partisan efficiency, the smallest Polsby--Popper score is now about $0.12$.
For reference, Case~\cite{Case:07} suggests flagging scores below $0.20$ as instances of possible gerrymandering.
Our main result (Theorem~\ref{thm.main result}) establishes that this tradeoff is unavoidable with any districting system.
}
\end{figure}

\section*{Acknowledgments}
This work was conceived from an episode of the podcast More Perfect while DGM was visiting New York University to speak in the Math and Data Seminar.
The authors thank Joey Iverson and John Jasper for enlightening conversations, as well as Mira Bernstein, Ben Blum-Smith, Moon Duchin and Vladimir Kogan for constructive feedback that substantially improved the discussion of our results.
DGM was partially supported by AFOSR F4FGA06060J007 and AFOSR Young Investigator Research Program award F4FGA06088J001.
The views expressed in this article are those of the authors and do not reflect the official policy or position of the authors' employers, the United States Air Force, Department of Defense, or the U.S.\ Government.

\end{document}